\documentclass[12pt]{article}
\usepackage{amsxtra}
\usepackage{amsthm}
\usepackage{amsmath}
\usepackage{amsfonts}
\usepackage{verbatim}
\usepackage{amssymb}
\usepackage{palatino, url, multicol}
\newcommand{\Z}{{\mathbb Z}}
\newcommand{\sll}{\widehat{\mathfrak sl}_\ell}
\newcommand{\affS}[1]{\widehat{S_{#1}}}

\newtheorem{theorem}{Theorem}[subsection]
\newtheorem{corollary}[theorem]{Corollary}
\newtheorem{proposition}[theorem]{Proposition}
\newtheorem{lemma}[theorem]{Lemma}

\theoremstyle{definition}
\newtheorem{definition}[theorem]{Definition}
\newtheorem{example}[theorem]{Example}

\theoremstyle{remark}
\newtheorem{remark}[theorem]{Remark}
\numberwithin{equation}{section}

\title{ $(\ell,0)$-Carter partitions, their crystal-theoretic behavior and generating function}

\author{Chris Berg \thanks{Supported in part by NSF grant DMS-0135345}\\
\small Department of Mathematics\\[-0.8ex]
\small Davis, CA, 95616 USA\\[-0.8ex]
\small \texttt{berg@math.ucdavis.edu} \and Monica Vazirani\thanks{Supported in part by NSF grant DMS-0301320}\\
\small Department of Mathematics\\[-0.8ex]
\small Davis, CA, 95616 USA\\[-0.8ex]
\small \texttt{vazirani@math.ucdavis.edu}}

\begin{document}
\maketitle

\newlength\cellsize \setlength\cellsize{18\unitlength}
\savebox2{%
\begin{picture}(18,18)
\put(0,0){\line(1,0){18}}
\put(0,0){\line(0,1){18}}
\put(18,0){\line(0,1){18}}
\put(0,18){\line(1,0){18}}
\end{picture}}
\newcommand\cellify[1]{\def\thearg{#1}\def\nothing{}%
\ifx\thearg\nothing
\vrule width0pt height\cellsize depth0pt\else
\hbox to 0pt{\usebox2\hss}\fi%
\vbox to 18\unitlength{
\vss
\hbox to 18\unitlength{\hss$#1$\hss}
\vss}}
\newcommand\tableau[1]{\vtop{\let\\=\cr
\setlength\baselineskip{-16000pt}
\setlength\lineskiplimit{16000pt}
\setlength\lineskip{0pt}
\halign{&\cellify{##}\cr#1\crcr}}}
\savebox3{%
\begin{picture}(15,15)
\put(0,0){\line(1,0){15}}
\put(0,0){\line(0,1){15}}
\put(15,0){\line(0,1){15}}
\put(0,15){\line(1,0){15}}
\end{picture}}
\newcommand\expath[1]{%
\hbox to 0pt{\usebox3\hss}%
\vbox to 15\unitlength{
\vss
\hbox to 15\unitlength{\hss$#1$\hss}
\vss}}


\begin{abstract}
In this paper we give an alternate combinatorial description of the ``$(\ell,0)$-Carter partitions'' (see \cite{F}). The representation-theoretic significance of these partitions is that they indicate the irreducibility of the corresponding specialized Specht module over the Hecke algebra of the symmetric group (see \cite{JM}). Our main theorem is the equivalence of our combinatoric and the one introduced by James and Mathas (\cite{JM}), which is in terms of hook lengths.  
We use our result to find a generating series which counts such partitions, with respect to the statistic of a partition's first part. We then apply our description of these partitions to the crystal graph $B(\Lambda_0)$ of the basic representation of $\widehat{\mathfrak{sl}_{\ell}}$, whose nodes are labeled by $\ell$-regular partitions. Here we give a fairly simple crystal-theoretic rule which generates all  $(\ell,0)$-Carter partitions in the graph $B(\Lambda_0)$.
\end{abstract}

\section{Introduction}
\subsection{Preliminaries}\label{prelim}
Let $\lambda$ be a partition of $n$ and $\ell \geq 2$ be an integer. We will use the convention $(x,y)$ to denote the box which sits in the $x^{\textrm{th}}$ row and the $y^{\textrm{th}}$ column of the Young diagram of $\lambda$. Throughout this paper, all of our partitions are drawn in English notation. $\mathcal{P}$ will denote the set of all partitions. An \textit{$\ell$-regular partition} is one in which no nonzero part occurs $\ell$ or more times. The length of a partition $\lambda$ is defined to be the number of nonzero parts of $\lambda$ and is denoted $len(\lambda)$.

The \textit{hook length} of the $(a,c)$ box of $\lambda$ is defined to be the number of boxes to the right of or below the box $(a,c)$, including the box $(a,c)$ itself. It is denoted \textit{$h_{(a,c)}^{\lambda}$}.

The rim of $\lambda$ are those boxes at the ends of their rows or columns. An \textit{$\ell$-rim hook} is a connected sequence of $\ell$ boxes in the rim. It is \textit{removable} if when it is removed from $\lambda$, the remaining diagram is the Young diagram of some other (non-skew) partition. To lighten notation, we will abbreviate and call a removable $\ell$-rim hook an $\ell$-rim hook.

A partition which has no removable $\ell$-rim hooks is called an \textit{$\ell$-core}. The set of all $\ell$-cores is denoted $\mathcal{C}_{\ell}$. 

\begin{remark}\label{divisibility}
A necessary and sufficient condition that $\lambda$ be an $\ell$-core is that $\ell \nmid h_{(a,c)}^\lambda$ for all $(a,c) \in \lambda$ (see \cite{JK}).
\end{remark}

 Every partition has a well defined $\ell$-core, which is obtained by successively removing any possible $\ell$-rim hooks. The $\ell$-core is uniquely determined from the partition, independently of choice of  the order in which one successively removes $\ell$-rim hooks. The number of $\ell$-rim hooks which must be removed from a partition $\lambda$ to obtain its core is called the \textit{weight} of $\lambda$. See \cite{JK} for more details.

Removable $\ell$-rim hooks whose boxes all sit in one row will be called \textit{horizontal $\ell$-rim hooks}. Equivalently, they are also commonly called $\ell$-rim hooks with leg length 0, or $\ell$-ribbons with spin 0. Removable $\ell$-rim hooks which are not horizontal will be called \textit{non-horizontal $\ell$-rim hooks}.
\begin{definition}\label{lpartition}
An \textit{$\ell$-partition} is a partition $\lambda$ such that:
\begin{itemize}
\item $\lambda$ has no non-horizontal $\ell$-rim hooks;
\item when any number of horizontal $\ell$-rim hooks are removed from $\lambda$, the remaining diagram has no non-horizontal $\ell$-rim hooks. 
\end{itemize}
\end{definition}
We remark that an $\ell$-partition is necessarily $\ell$-regular.
\begin{example}  Any $\ell$-core is also an $\ell$-partition. \end{example}
\begin{example}
$(5,4,1)$ is a 6-core, hence a 6-partition. It is a 2-partition, but not a 2-core. It is not a 3-, 4-, 5- or 7-partition. It is an $\ell$-core for $\ell > 7$.
$$ 
	\tableau{ \mbox{} &  \mbox{} &  \mbox{} &  \mbox{} &  \mbox{}\\
	 \mbox{} &  \mbox{} &  \mbox{} &  \mbox{}\\
	 \mbox{}\\}
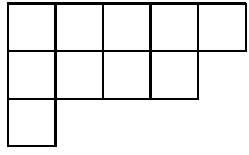
$$
\end{example}
To understand the representation-theoretic significance of $\ell$-partitions, it is necessary to introduce the Hecke algebra of the symmetric group.

\begin{definition}
For a fixed field $\mathbb{F}$ and $0 \neq q \in \mathbb{F}$, the finite Hecke algebra $H_n(q)$ is defined to be the algebra over $\mathbb{F}$ generated by $T_1, ... ,T_{n-1}$ with relations

$$\begin{array}{ll} 	T_i  T_j = T_j  T_i    				& \textrm{for $|i-j|>1$}\\
			T_i  T_{i+1}  T_i = T_{i+1}  T_i  T_{i+1} 	& \textrm{for $i < n-1$}\\
			T_i^2 = (q-1)T_i + q				& \textrm{for $i \leq n-1$}.\\

\end{array}$$
\end{definition}
In this paper we will always assume that $q \neq 1$, that $q \in \mathbb{F}$ is a primitive $\ell^{th}$ root of unity (so necessarily $\ell \geq 2$) and that the characteristic of $\mathbb{F}$ is zero.

Similar to the symmetric group, a construction of the Specht module $S^{\lambda} = S^{\lambda}[q]$ exists for $H_n(q)$ (see \cite{DJ}).
For $k \in \mathbb{Z}$, let
$$\nu_{\ell}(k) = \left\{ 	\begin{array}{ll}
			  1 &  \textrm{ $\ell \mid k$}\\
			  0 &  \textrm{ $\ell \nmid k$}. 
			\end{array} \right.$$ 
It is known that the Specht module $S^{\lambda}$ indexed by an $\ell$-regular partition $\lambda$  is irreducible  if and only if 
$$\begin{array}{lccrr} (\star) &  \nu_{\ell}(h_{(a,c)}^\lambda) = \nu_{\ell}(h_{(b,c)}^\lambda) & \textrm{for all pairs $(a,c)$, $(b,c) \in \lambda$} \end{array}$$
(see \cite{JM} Theorem 4.12). Partitions which satisfy $(\star)$ have been called in the literature $(\ell,0)$-Carter partitions.    So, a necessary and sufficient condition for the irreducibility of the Specht 
module indexed by an $\ell$-regular partition is that the hook lengths in a column of the partition $\lambda$ are either all divisible by $\ell$ or none of them are, for every column (see \cite{F}  for general partitions, when $\ell \geq 3$).

We remark that a Specht module $S^\lambda$ is both irreducible and projective
if and only if $\lambda$ is an $\ell$-core (one can easily see that the characterization of $\ell$-cores given in Remark \ref{divisibility} is a stronger condition than $(\star)$).

All of the irreducible representations of $H_n(q)$ have been constructed when $q$ is a primitive $\ell^{th}$ root of unity. For $\ell$-regular $\lambda$, $S^\lambda$ has a unique simple quotient, denoted $D^{\lambda}$, and all simples can be obtained in this way (see \cite{DJ} for more details). In particular $D^{\lambda} = S^{\lambda}$ if and only if $S^{\lambda}$ is irreducible and $\lambda$ is $\ell$-regular.

Let $\nu_{p}'(k) = max \{ m : \,\, p^m \mid k \} $. 
In the symmetric group setting, for a prime $p$, the requirement for the irreducibility of the Specht module indexed by a $p$-regular partition over the field $\mathbb{F}_p$ is that 
$$\begin{array}{ccrr}  \nu_{p}'(h_{(a,c)}^\lambda) = \nu_{p}'(h_{(b,c)}^\lambda) & \textrm{for all pairs $(a,c)$, $(b,c) \in \lambda$} \end{array}$$ (see \cite{JK}).

Note that $\nu_{\ell}$ is related to $\nu_{\ell}'$ in that $\nu_{\ell}(k) = \max \{m : [\ell]_z^m \mid [k]_z \}$, where $z$ is an indeterminate and $[k]_z = \frac{z^{k}-1}{z-1} \in \mathbb{C}[z]$.

From Example 2, we can see that $S^{(5,4,1)}$ is irreducible over $H_{10}(-1)$, but it is 
reducible over $\mathbb{F}_2 S_{10}$. This highlights how the problem of determining the irreducible Specht modules is different for $\mathbb{F}_p S_n$ and $H_n(q)$ where $q = e^{\frac{2 \pi i} {p}}$. This paper restricts its attention to $H_n(q)$.

Because $(\ell,0)$-Carter partitions have a significant representation-the\-o\-retic interpretation,
it is natural to ask if these partitions exhibit interesting behavior 
in the crystal graph of the basic representation of $\sll$.
This crystal is a combinatorial object that, 
in addition to describing the basic representation,
 parameterizes the irreducible
representations of $H_n(q)$, $n \ge 0$ and encodes various
representation-theoretic subtleties.  
The nodes of the crystal can be labeled by $\ell$-regular partitions
and edges encode partial information about the functors of restriction and
induction. 

By way of analogy, 
in the crystal the $\ell$-cores are exactly the extremal nodes,
or in other words given by the orbit of the highest weight node
under the action of $\affS{\ell}$, the affine symmetric group.
The $(\ell,0)$-Carter partitions do not behave as nicely with respect to the
$\affS{\ell}$-action, but do share many similarities with 
$\ell$-cores from this point of view. The theorems of Section \ref{crystal}
explain precisely how.

We remark that the crystal does not depend on the characteristic of
the underlying field that $H_n(q)$ is defined over,
but the characterization of $(\ell,0)$-Carter partitions does.
Thus we expect some inherent asymmetry in the behavior 
of these partitions in the crystal, which we indeed see.
The pattern was also interesting in its own right, so worth including
just for this consideration. 

\subsection{Outline}
Here we summarize the main results of this paper. Section \ref{new_def_ell_part} shows the equivalence of $\ell$-partitions and $(\ell,0)$-Carter partitions (see Theorem \ref{maintheorem}). Section \ref{?} gives a different classification of $\ell$-partitions  which allows us to give an explicit formula for a generating function for the number of $\ell$-partitions with respect to the statistic of a partition's first part.  In Section \ref{crystal} we describe the crystal-theoretic behavior of $\ell$-partitions. There we explain where in the crystal graph $B(\Lambda_0)$ one can expect to find $\ell$-partitions (see Theorems \ref{top_and_bottom},  \ref{other_cases} and \ref{second_from_bottom}). 
Section \ref{new_proof} gives a representation-theoretic proof of Theorem \ref{top_and_bottom}.
Finally, in Section \ref{conclusion}, we mention how our results can be generalized  to all  Specht modules (not necessarily those indexed by $\ell$-regular partitions) which stay irreducible at a primitive $\ell^{th}$ root of unity (for $\ell >2$), which relies on recent results of Fayers (see \cite{F}) and Lyle (see \cite{L}).

\section{$\ell$-partitions}\label{new_def_ell_part}
In this section, we prove that a partition is an $\ell$-partition if and only if it satisfies $(\star)$. To prove this, we will first need two lemmas which tell us when we can add/remove horizontal $\ell$-rim hooks to/from a diagram. Henceforth, we will no longer use the term ``$(\ell,0)$-Carter partition" when referring to condition $(\star)$. 
\subsection{Equivalence of the combinatorics}
\begin{lemma}\label{lemma1}
Suppose $\lambda$ is a partition which does not satisfy ($\star$), and that $\mu$ is a partition obtained by adding a horizontal $\ell$-rim hook to $\lambda$. Then $\mu$ does not satisfy ($\star$).
\end{lemma}
\begin{proof} If $\lambda$ does not satisfy ($\star$), it means that somewhere in the partition there are two boxes $(a,c)$ and $(b,c)$ with $\ell$ dividing exactly one of $h_{(a,c)}^{\lambda}$ and $h_{(b,c)}^{\lambda}$. We will assume $a < b$. Here we prove the lemma in the case where $\ell \mid h_{(a,c)}^{\lambda}$ and $\ell \nmid h_{(b,c)}^{\lambda}$, the other case being similar. 

\subsubsection*{Case 1}
It is easy to see that adding a horizontal $\ell$-rim hook in row $i$ for $i < a$ or $a<i<b$ will not change the hook lengths in the boxes $(a,c)$ and $(b,c)$. In other words, $h_{(a,c)}^{\lambda}$ = $h_{(a,c)}^{\mu}$ and $h_{(b,c)}^{\lambda} = h_{(b,c)}^{\mu}$. 

\subsubsection*{Case 2}
 If the horizontal $\ell$-rim hook is added to row $a$, then $h_{(a,c)}^{\lambda} + \ell = h_{(a,c)}^{\mu}$ and $h_{(b,c)}^{\lambda} = h_{(b,c)}^{\mu}$.  Similarly if the new horizontal $\ell$-rim hook is added in row $b$, $h_{(a,c)}^{\lambda} = h_{(a,c)}^{\mu}$ and $h_{(b,c)}^{\lambda} + \ell = h_{(b,c)}^{\mu}$. Still, $\ell \mid h_{(a,c)}^{\mu}$ and $\ell \nmid h_{(b,c)}^{\mu}$.

\subsubsection*{Case 3}
 Suppose the horizontal $\ell$-rim hook is added in row $i$ with $i>b$. If the box $(i,c)$ is not in the added $\ell$-rim hook then $h_{(a,c)}^{\lambda}$ = $h_{(a,c)}^{\mu}$ and $h_{(b,c)}^{\lambda} = h_{(b,c)}^{\mu}$. If the box $(i,c)$ is in the added $\ell$-rim hook, then there are two sub-cases to consider. 
 If $(i,c)$ is the rightmost box of the added $\ell$-rim hook then $\ell \mid h_{(a,c- \ell + 1)}^{\mu}$ and $\ell \nmid h_{(b,c- \ell + 1)}^{\mu}$. 
 Otherwise $(i,c)$ is not at the end of the added $\ell$-rim hook, in which case $\ell \mid h_{(a,c + 1)}^{\mu}$ and $\ell \nmid h_{(b,c + 1)}^{\mu}$. 
  In all cases, $\mu$ does not satisfy $(\star)$.
\end{proof}

\begin{example}

 Let $\lambda$ = $(14,9,5,2,1)$ and $\ell = 3$. This partition does not satisfy ($\star$). For instance, looking at boxes $(2,3)$ and $(3,3)$ highlighted below, we see that $3 \mid h_{(3,3)}^{\lambda} =3$ but $3 \nmid h_{(2,3)}^{\lambda} = 8$. Let $\lambda[i]$ denote the partition obtained when adding a horizontal $\ell$-rim hook to the $i^{th}$ row of $\lambda$ (when it is still a partition). Adding a horizontal $3$-rim hook in row 1 will not change  $h_{(2,3)}^{\lambda}$ or $h_{(3,3)}^{\lambda}$ (Case 1 of Lemma \ref{lemma1}).  Adding a horizontal 3-rim hook to row 2 will make $h_{(2,3)}^{\lambda[2]} = 11$, which is congruent to $h_{(2,3)}^{\lambda}$ modulo 3 (Case 2 of Lemma \ref{lemma1}). Adding in row 3 is also Case 2. Adding a horizontal 3-rim hook to row 4 will make $h_{(2,3)}^{\lambda[4]} = 9$ and $h_{(3,3)}^{\lambda[4]}=4$, but one column to the right, we see that now $h_{(2,4)}^{\lambda[4]} = 8$ and $h_{(3,4)}^{\lambda[4]} = 3$ (Case 3 of Lemma \ref{lemma1}).
$$\begin{array}{cc} \tableau{18 & 16 & 14 & 13 & 12 & 10 & 9 & 8 & 7 & 5 & 4 & 3 & 2 & 1\\
12&10&8&7&6&4&3&2&1\\
7&5&3&2&1\\
3&1\\
1} & 
\linethickness{2.5pt}
\put (-226,-37){\line(0,1){37}}
\put (-208,-37){\line(0,1){37}}
\put (-226,-36){\line(1,0){19}}
\put (-226,-18){\line(1,0){19}}

\put (-226,-1){\line(1,0){19}}

\end{array}$$
\end{example}

\begin{lemma}\label{lemma2}
Suppose ${\lambda}$ does not satisfy ($\star$). Let $a,b,c$ be such that $\ell$ divides exactly one of $h_{(a,c)}^{\lambda}$ and $h_{(b,c)}^{\lambda}$ with $a<b$. Suppose $\nu$ is a partition obtained from $\lambda$ by removing a horizontal $\ell$-rim hook, and that $(b,c) \in \nu$. Then $\nu$ does not satisfy ($\star$).\end{lemma}
As the proof of Lemma \ref{lemma2} is similar to that of Lemma \ref{lemma1}, we leave it to the reader.

\begin{remark}\label{remark1} In the proof of Lemma \ref{lemma1} we have also shown that when adding a horizontal $\ell$-rim hook to a partition which does not satisfy $(\star)$, the violation to ($\star$) occurs in the same rows as in the original partition. It can also be shown in Lemma \ref{lemma2} that when removing a horizontal $\ell$-rim hook (in the cases above), the violation will stay in the same rows as in the original partition. This will be useful in the proof of Theorem \ref{maintheorem}.
\end{remark}
\begin{example} We illustrate here the necessity of our hypothesis that $(b,c) \in \nu$.
$\lambda = (5,4,1)$ does not satisfy ($\star$) for $\ell=3$. The boxes $(1,2)$ and $(2,2)$ are a violation of $(\star)$. Removing a horizontal 3-rim hook will give the partition $\nu = (5,1,1)$ which does satisfy ($\star$). Note that this does not violate Lemma \ref{lemma2}, since $\nu$ does not contain the box $(2,2)$.
$$\begin{array}{cccc}
	\tableau{7&5&4&3&1\\
		5&3&2&1\\
		1}&$\:\:\:\:\:\:\:\:\,\,\,\,$
		&
	\tableau{7&4&3&2&1\\
         		2\\
		1}
		&
		\linethickness{2.5pt}
		\put (-222,18){\line(1,0){18}}		
		\put (-222,0){\line(1,0){18}}
		\put (-222,-18){\line(1,0){18}}
		\put (-222,-18){\line(0,1){36}}
		\put (-204,-18){\line(0,1){36}}

		\end{array}
$$
\end{example}

\begin{theorem}\label{maintheorem}
A partition is an $\ell$-partition if and only if it satisfies ($\star$).
\end{theorem}
\begin{proof} 
{\sloppy Suppose $\lambda$ is not an $\ell$-partition. We may remove horizontal $\ell$-rim hooks from $\lambda$ until we obtain a partition $\mu$ which has a non-horizontal $\ell$-rim hook.

}

{\sloppy
We label the upper rightmost box of the non-horizontal $\ell$-rim hook $(a,c)$ and lower leftmost box $(b,d)$ with $a < b$. 
Then $h_{(a,d)}^{\mu} = \ell$ and $h_{(b,d)}^{\mu} < \ell$, so $\mu$ does not satisfy ($\star$). 
From Lemma \ref{lemma1}, since $\lambda$ is obtained from $\mu$ by adding horizontal $\ell$-rim hooks, $\lambda$ also does not satisfy ($\star$).
}

Conversely, suppose  $\lambda$ does not satisfy ($\star$). Let $(a,c), (b,c) \in \lambda$ be such that $\ell$ divides exactly one of $h_{(a,c)}^{\lambda}$ and $h_{(b,c)}^{\lambda}$.  Let us assume that $\lambda$ is an $\ell$-partition and we will derive a contradiction.

\subsubsection*{Case 1} Suppose that $a<b$ and that $\ell \mid h_{(a,c)}^{\lambda}$. Then without loss of generality we may assume that $b=a+1$.  By the equivalent characterization of $\ell$-cores mentioned in Section \ref{prelim}, there exists at least one removable $\ell$-rim hook in $\lambda$ . By assumption it must be horizontal.
 If an $\ell$-rim hook exists which does not contain the box $(b,c)$ then let $\lambda^{(1)}$ be $\lambda$ with this $\ell$-rim hook removed. By Lemma \ref{lemma2}, since we did not remove the $(b,c)$ box, $\lambda^{(1)}$ will still not satisfy ($\star$).  Then there are boxes $(a, c_1)$ and $(b, c_1)$ for which $\ell \mid h_{(a,c_1)}^{\lambda^{(1)}}$ but $\ell \nmid h_{(b,c_1)}^{\lambda^{(1)}}$. 
By Remark \ref{remark1} above, we can assume that the violation to $(\star)$ is in the same rows $a$ and $b$ of $\lambda^{(1)}$. We apply the same process as above repeatedly until we must remove a horizontal $\ell$-rim hook from the partition $\lambda^{(k)}$ which contains the $(b,c_k)$ box, and in particular we cannot remove a horizontal $\ell$-rim hook from row $a$. Let $d$ be so that $h_{(b,d)} = 1$. Such a $d$ must exist since we can remove a horizontal $\ell$-rim hook from this row. Since $(b,c_k)$ is removed from $\lambda^{(k)}$ when we remove the horizontal $\ell$-rim hook, $h_{(b,c_k)}^{\lambda^{(k)}}  < \ell$ ($\ell$ does not divide $h_{(b,c_k)}^{\lambda^{(k)}}$ by assumption, so in particular $h_{(b,c_k)}^{\lambda^{(k)}} \neq \ell$). Note that $h_{(a,c_k)}^{\lambda^{(k)}} = h_{(b,c_k)}^{\lambda^{(k)}} + h_{(a,d)}^{\lambda^{(k)}} -1$, $\ell \mid h_{(a,c_k)}^{\lambda^{(k)}}$ and $\ell \nmid h_{(b,c_k)}^{\lambda^{(k)}}$, so $\ell \nmid (h_{(a,d)}^{\lambda^{(k)}} -1)$. If $h_{(a,d)}^{\lambda^{(k)}} -1 > \ell$ then we could remove a horizontal $\ell$-rim hook from row $a$, which we cannot do by assumption. Otherwise $h_{(a,d)}^{\lambda^{(k)}} < \ell$. Then a non-horizontal $\ell$-rim hook exists starting at the rightmost box of the $a^{th}$ row, going left to $(a,d)$, down to $(b,d)$ and then left. This is a contradiction as we have assumed that $\lambda$ was an $\ell$-partition.
\subsubsection*{Case 2} Suppose that $a<b$ and that $\ell \mid h_{(b,c)}^{\lambda}$. We will reduce this to Case 1. Without loss of generality we may assume that $b = a+1$ and that $\ell \mid h_{(n,c)}^{\lambda}$ for all $n>a$, since otherwise we are in Case 1. Let $m$ be so that $(m,c) \in \lambda$ but $(m+1,c) \not\in \lambda$. Then because $h_{(m,c)}^{\lambda} \geq \ell$, the list $h_{(a,c)}^{\lambda}, h_{(a, c+1)}^{\lambda} = h_{(a,c)}^{\lambda} -1 , \ldots ,h_{(a, c+\ell-1)}^{\lambda} = h_{(a,c)}^{\lambda} -\ell+1$ consists of $\ell$ consecutive integers. Hence one of them must be divisible by $\ell$. Suppose it is $h_{(a, c+i)}^{\lambda}$. Note $\ell \nmid h_{(m, c+i)}^{\lambda}$, since $h_{(m, c+i)}^{\lambda} = h_{(m,c)}^{\lambda} - i$ and $\ell \mid h_{(m,c)}^{\lambda}$. Then we may apply Case 1 to the boxes $(a,c+i)$ and $(m,c+i)$.

\end{proof}

\begin{remark} This result can actually be obtained using a more general result of James and Mathas (\cite{JM}, Theorem 4.20), where they classified which $S^{\lambda}$ remain irreducible for $\lambda$ $\ell$-regular. However, we have included this proof to emphasize the simplicity of the theorem and its simple combinatorial proof in this context. \end{remark}

\begin{remark}  When $q$ is a primitive $\ell^{th}$ root of unity, and $\lambda$ is an $\ell$-regular partition, the Specht module $S^{\lambda}$ of $H_n(q)$ is irreducible if and only if  $\lambda$ is an $\ell$-partition.
This follows from what was said above concerning the James and Mathas result on  the equivalence of ($\star$) and irreducibility of Specht modules, and Theorem \ref{maintheorem}.
\end{remark}

\section{Generating functions}\label{?}
Let $\mathcal{L}_{\ell}$ denote the set of $\ell$-partitions. In this section, we study the generating function of $\ell$-partitions with respect to the statistic of the first part of the partition. We thank Richard Stanley for suggesting that we compute the generating function.
\subsection{Counting $\ell$-cores}
We will count $\ell$-cores first, with respect to the statistic of the first part of the partition. Let $$C_{\ell}(x) = \sum_{k=0}^{\infty} c_k^{\ell} x^k$$ where $c_k^{\ell} =  \#\{ \lambda \in \mathcal{C}_{\ell} : \lambda_1 = k\}$. Note that this does not depend on the size of the partition, only its first part. Also, the empty partition is the unique partition with first part $0$, and is always an  $\ell$-core, so that $c_0^{\ell} =1$ for every $\ell$.

\begin{example}\label{2_cores}

For $\ell = 2$, all 2-cores are of the form $\lambda = (k,k-1,\dots,2,1)$. Hence $C_2 (x) = \sum_{k=0}^{\infty} x^k = \frac{1}{1-x}$.
\end{example}
\begin{example}
For $\ell = 3$, the first few cores are $$\emptyset, (1), (1,1), (2), (2,1,1), (2,2,1,1), \dots $$ so $C_3(x) = 1+ 2x +  3x^2 + \hdots$
\end{example}

For a partition $\lambda = (\lambda_1,...,\lambda_s)$ with $\lambda_s>0$, the \textit{$\beta$-numbers }$(\beta_1,...,\beta_s)$ of $\lambda$ are defined to be the hook lengths of the first column (i.e. $\beta_i = h_{(i,1)}^{\lambda}$). Note that this is a modified version of the $\beta$-numbers defined by James and Kerber in \cite{JK}, where all definitions in this section can be found. We draw a diagram $\ell$ columns wide with the numbers $\{0,1,2,\dots,\ell-1\}$ inserted in the first row in order, $\{ \ell, \ell +1, \dots, 2\ell-1\}$ inserted in the second row in order, etc. Then we circle all of the $\beta$-numbers for $\lambda$. The columns of this diagram are called \textit{runners}, the circled numbers are called \textit{beads}, the uncircled numbers are called \textit{gaps}, and the diagram is called an \textit{abacus} . It is well known that a partition $\lambda$ is an $\ell$-core if and only if all of the beads lie in the last $\ell -1$ runners and there is no gap above any bead.
\begin{example}

 $\lambda = (4,2,2,1,1)$ has $\beta$-numbers $8,5,4,2,1$.  In the abacus for $\ell =3$ the first runner is empty, the second runner has beads at 1 and 4, and the third runner has beads at 2, 5 and 8 (as pictured below). Hence $\lambda$ is a $3$-core.

\begin{multicols}{2}{
$\begin{array}{lr}
\lambda = &
\tableau{8 & 5 & 2 & 1\\
	5 & 2\\
	4 & 1\\
	2\\
	1}
\end{array}
$

\begin{center}
\begin{picture}(100,70)
\put (10,65) {0}
\put (40,65){1}
\put (70,65){2}
\put (10,50){3}
\put (40, 50){4}
\put (70,50){5}
\put (10,35){6}
\put (40,35){7}
\put (70,35){8}
\put (10,20){9}
\put (38,20){10}
\put (68,20){11}
\put (72.5,53){\circle{13}}
\put (72.5,68){\circle{13}}
\put (72.5,38){\circle{13}}
\put (42.5,68){\circle{13}}
\put (42.5,53){\circle{13}}
\put (42.5,12){.}
\put (42.5,8){.}
\put (42.5,4){.}
\put (72.5,12){.}
\put (72.5,8){.}
\put (72.5,4){.}
\put (11.5,12){.}
\put (11.5,8){.}
\put (11.5,4){.}
\end{picture}
\end{center}}
\end{multicols}
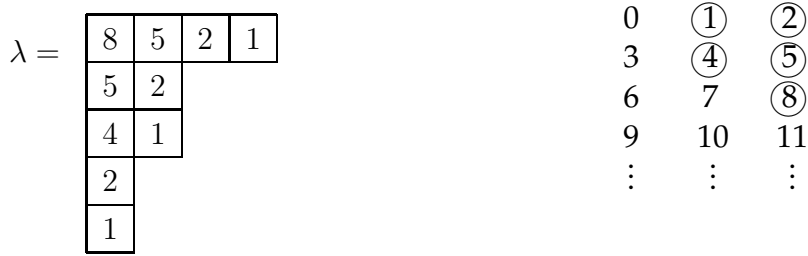
$\begin{array}{c} \textrm{Young diagram and abacus of $\lambda = (4,2,2,1,1)$ } \end{array}$
\end{example}
\begin{proposition}\label{ell_core_bijection}
There is a bijection between the set of $\ell$-cores with first part $k$ and the set of $(\ell-1)$-cores with first part $\leq k$.
\end{proposition}
\begin{proof}
Using the abacus description of cores, we describe our bijection as follows: 

Given an $\ell$-core with largest part $k$, remove the whole runner which contains the largest bead (the bead with the largest $\beta$-number). In the case that there are no beads, remove the rightmost runner. The remaining runners can be placed into an $\ell-1$ abacus in order. The remaining abacus will clearly have its first runner empty. This will correspond to an $(\ell - 1)$-core with largest part at most $k$. This map gives a bijection between the set of all $\ell$-cores with largest part $k$ and the set of all $(\ell-1)$-cores with largest part at most $k$.

To see that it is a bijection, we will give its inverse. Given the abacus for an $(\ell-1)$-core $\lambda$ and a $k \geq \lambda_1$, insert the new runner directly after the $k^{th}$ gap, placing a bead on it directly after the $k^{th}$ gap and at all places above that bead on the new runner.

\end{proof}

\begin{corollary}\label{ell_core_binom}
$c_k^{\ell} = \binom{k+\ell -2}{k}.$
\end{corollary}
\begin{proof}  This proof is by induction on $\ell$. For $\ell = 2$, as the only $1$-core is the empty partition, by Proposition \ref{ell_core_bijection}\, $c_k^2 = 1 = \binom{k}{k}$. Note this was also observed in Example \ref{2_cores}. For the rest of the proof, we assume that $\ell >2$. 

It follows directly from Proposition \ref{ell_core_bijection} that $$(\sharp) \,\,\,\,\,\,\,c_k^{\ell} = \sum_{j=0}^k c_j^{\ell -1}.$$ 
Recall the fact that $\binom{\ell+k-2}{k} = \binom{\ell-3}{0} + \binom{\ell-2}{1}+ \dots + \binom{\ell+k-3}{k}$ for $\ell>2$. 
Applying our inductive hypothesis to all of the terms in the right hand side of $(\sharp)$ we get that  $c_k^\ell=  \sum_{j=0}^k c_j^{\ell-1}  = \sum_{j=0}^k \binom{\ell+j-3}{j}=  \binom{\ell+k-2}{k}$. 
Therefore, the set of all $\ell$-cores with largest part $k$ has cardinality $\binom{k+\ell-2}{k}$.
\end{proof}

\begin{remark}
The bijection above between $\ell$-cores with first part $k$ and $(\ell -1)$-cores with first part $\leq k$ has several other descriptions, using different interpretations of $\ell$-cores. Together with Brant Jones, we have a paper on some of these descriptions. See \cite{BJV} for more details.
\end{remark}

\begin{example}

Let $\ell = 3$ and $\lambda = (4,2,2,1,1)$. The abacus for $\lambda$ is:
 
\begin{center}
\begin{picture}(80,80)
\put (10,65) {0}
\put (40,65){1}
\put (70,65){2}
\put (10,50){3}
\put (40, 50){4}
\put (70,50){5}
\put (10,35){6}
\put (40,35){7}
\put (70,35){8}
\put (10,20){9}
\put (38,20){10}
\put (68,20){11}
\put (72.5,53){\circle{13}}
\put (72.5,68){\circle{13}}
\put (72.5,38){\circle{13}}
\put (42.5,68){\circle{13}}
\put (42.5,53){\circle{13}}
\put (42.5,12){.}
\put (42.5,8){.}
\put (42.5,4){.}
\put (72.5,12){.}
\put (72.5,8){.}
\put (72.5,4){.}
\put (11.5,12){.}
\put (11.5,8){.}
\put (11.5,4){.}
\end{picture}
\end{center}
The largest $\beta$-number is 8. Removing the whole runner in the same column as the 8, we get the remaining diagram with runners relabeled for $\ell = 2$
\begin{center}
\begin{picture}(80,80)
\put (10,65) {0}
\put (40,65){1}
\put (70,65){}
\put (10,50){2}
\put (40, 50){3}
\put (70,50){}
\put (10,35){4}
\put (40,35){5}
\put (70,35){}
\put (10,20){6}
\put (40,20){7}
\put (68,20){}
\put (72.5,53){\circle{13}}
\put (72.5,68){\circle{13}}
\put (72.5,38){\circle{13}}
\put (42.5,68){\circle{13}}
\put (42.5,53){\circle{13}}
\Large
\put (66,64){$\times$}
\put (66,49){$\times$}
\put (66,34){$\times$}
\put (42,12){.}
\put (42,8){.}
\put (42,4){.}
\put (10.5,12){.}
\put (10.5,8){.}
\put (10.5,4){.}
\end{picture}.
\end{center}

This is the abacus for the partition $(2,1)$, which is a 2-core with largest part $\leq 4$. 
\end{example}
 From Corollary \ref{ell_core_binom} , we obtain $C_{\ell}(x) = \sum_{k \geq 0} \binom{k+\ell-2}{k} x^k$ and so conclude the following.
 
 \begin{proposition}
 $$C_{\ell}(x) = \frac{1}{(1-x)^{\ell -1}}.$$
\end{proposition}

\subsection{Decomposing $\ell$-partitions}\label{construct}

We now describe a decomposition of $\ell$-partitions. We will use this to build $\ell$-partitions from $\ell$-cores and extend our generating function to $\ell$-partitions.

\begin{lemma}\label{newcores}
Let $\lambda$ be an $\ell$-core and $r>0$ an integer. Then 
\begin{enumerate}
\item $\nu = (\lambda_1+r(\ell-1), \lambda_1 + (r-1)(\ell-1), \dots, \lambda_1 + (\ell-1), \lambda_1, \lambda_2, \dots)$ is an $\ell$-core;

\item $\mu = (\lambda_2, \lambda_3, \dots)$ is an $\ell$-core.

\end{enumerate}
\end{lemma}

\begin{proof}
For $1 \leq i \leq r$, $\nu_i - \nu_{i+1} = \ell-1$, so the $i^{th}$ row of $\nu$ can never contain part of an $\ell$-rim hook. Because $\lambda$ is an $\ell$-core, $\nu$ cannot have an $\ell$-rim hook that is supported entirely on the rows below the $r^{th}$ row. Hence $\nu$ is an $\ell$-core.

For the second statement of the lemma, note the partition $\mu$ is simply $\lambda$ with its first row deleted. In particular, $h_{(a,b)}^\mu = h_{(a+1,b)}^\lambda$ for all $(a,b) \in \mu$, so that by Remark \ref{divisibility} it is an $\ell$-core.

\end{proof}


 We now construct a partition $\lambda$ from a triple of data $(\mu,r,\kappa)$ as follows.
 Let $\mu = (\mu_1, \dots , \mu_s)$ to be any $\ell$-core where $\mu_1 - \mu_2 \neq \ell - 1$. For an integer $r \geq 0$ we form a new $\ell$-core $\nu = (\nu_1, \dots \nu_r, \nu_{r+1}, \dots, \nu_{r+s})$ by 
attaching $r$ rows above $\mu$ so that:
$$\nu_r = \mu_1+ \ell-1,\, 
\nu_{r-1} = \mu_1+2(\ell-1),\,\dots \, ,\nu_1 = \mu_1 + r(\ell-1),$$ 
$$\nu_{r+i} = \mu_i \textrm{ for } i = 1, 2, \dots , s.$$
 By Lemma \ref{newcores}, $\nu$ is an $\ell$-core.

 Fix a partition $\kappa = (\kappa_1, \dots , \kappa_{r+1})$ with at most $(r+1)$ parts. Then the new partition $\lambda$ is obtained from $\nu$ by adding $\kappa_i$ horizontal $\ell$-rim hooks to row $i$ for every $i \in \{1, \dots, r+1\}$. In other words $\lambda_i = \nu_i + \ell \kappa_i$ for $i \in \{ 1, 2, \dots, r+1\}$ and $\lambda_i  = \nu_i$ for $i>r+1$. 
 
 From now on, when we associate $\lambda$ with the triple $(\mu,r, \kappa)$, we will think of $\mu \subset \lambda$ as embedded in the rows below the $r^{th}$ row in $\lambda$. We introduce the notation $\lambda \approx (\mu, r, \kappa)$ for this decomposition.

 \begin{theorem}\label{construction_theorem}
 Let $\mu$, $r$ and $\kappa$ be as above. Then $\lambda \approx (\mu,r,\kappa)$ is an $\ell$-partition. Conversely, every $\ell$-partition corresponds uniquely to a triple $(\mu,r,\kappa)$.
 \end{theorem}
 \begin{proof}
 
Suppose $\lambda \approx (\mu,r, \kappa)$ were not an $\ell$-partition. Then after removal of some number of horizontal $\ell$-rim hooks we obtain a partition $\rho$ which has a removable non-horizontal $\ell$-rim hook. Note that for $1 \leq i \leq r$, $\lambda_i-\lambda_{i+1} \equiv -1 \mod \ell$, and likewise  $\rho_i-\rho_{i+1} \equiv -1 \mod \ell$. Suppose the non-horizontal $\ell$-rim hook had its rightmost topmost box in the $j^{th}$ row of $\rho$. Necessarily it is the rightmost box in that row. Clearly we must have $j\leq r$ since $\mu$ is an $\ell$-core. If $\rho_j -\rho_{j+1} > \ell-1$ then this $\ell$-rim hook must lie entirely in the $j^{th}$ row, i.e. be horizontal. If $\rho_j - \rho_{j+1} = \ell-1$ then the $\ell$-rim hook is clearly not removable.
 
 Conversely, if $\lambda$ is an $\ell$-partition, then let $\kappa_i$ denote the number of removable horizontal $\ell$-rim hooks which must be removed from row $i$ to obtain the $\ell$-core $\nu$ of $\lambda$. Let $r$ denote the index of the first row for which $\nu_r - \nu_{r+1} \neq \ell-1$. Let $\mu = (\nu_{r+1}, \dots)$. Then $\lambda \approx (\mu,r,\kappa)$. 
 \end{proof}
 
\begin{example}

 For $\ell = 3$, $\mu = (2,1,1)$ is a 3-core with $\mu_1-\mu_2 \neq 2$. We may add three rows ($r=3$) to it to obtain $\nu = (8,6,4,2,1,1)$, which is still a 3-core. Now we may add three horizontal $\ell$-rim hooks to the first row, three to the second, one to the third and one to the fourth ($\kappa = (3,3,1,1)$) to obtain the partition $\lambda = (17,15,7,5,1,1)$, which is a 3-partition.
$$\begin{array}{cc}
	\mu =\,\,\,\,\,\,\,\, 
	\tableau{
	         4&1\\
		2\\
		1}\,\,\,\,\,\,\,\,\,&
\linethickness{2.5pt}
\put (40,-37){\line(1,0){37}}
\put (40,-92){\line(0,1){56}}
\put (39, -91){\line(1,0){20}}
\put (58,-92){\line(0,1){38}}
\put (58, -55){\line(1,0){19}}
\put (76,-55){\line(0,1){19}}
	\nu = \,\,\,\,\,\,\,\,\,\tableau{13&10&8&7&5&4&2&1\\
         		10&7&5&4&2&1\\
		7&4&2&1\\
		4&1\\
		2\\
		1} \end{array}
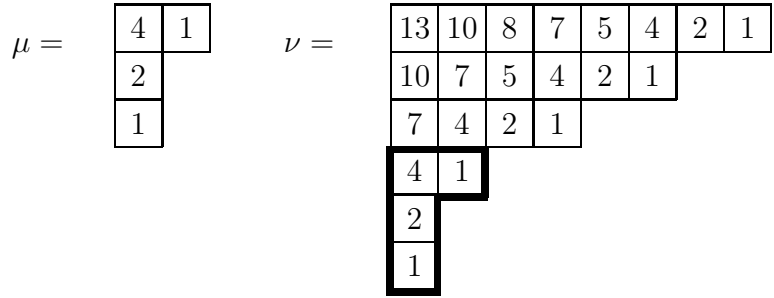
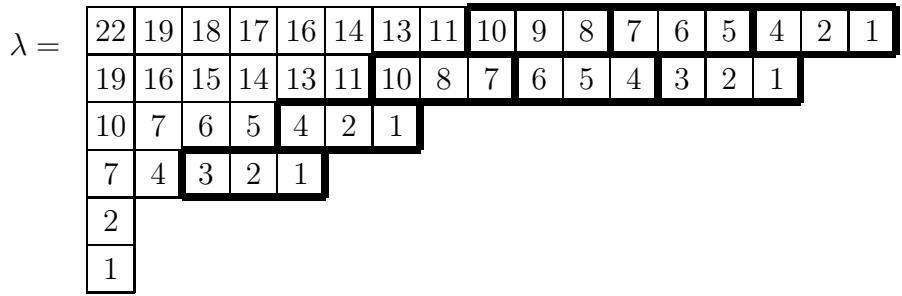
$$
$$\begin{array}{c} \mu \subset \nu \textrm{ is highlighted. } \end{array}$$
		
		$$\begin{array}{lr} \lambda = &
\linethickness{2.5pt} 
\put (306,0){\line(0,1){18}}
\put (252,0){\line(0,1){18}}
\put (198,0){\line(0,1){18}}
\put (144,0){\line(0,1){18}}
\put (270,-18){\line(0,1){18}}
\put (216,-18){\line(0,1){18}}
\put (162,-18){\line(0,1){18}}
\put (108,-18){\line(0,1){18}}
\put (126,-36){\line(0,1){18}}
\put (72,-36){\line(0,1){18}}
\put (90,-54){\line(0,1){18}}
\put (36,-54){\line(0,1){18}}
\put (36,-54){\line(1,0){54}}
\put (36,-36){\line(1,0){90}}
\put (72,-18){\line(1,0){198}}
\put (108,0){\line(1,0){198}}
\put (144,18){\line(1,0){162}}
	\tableau{22&19&18&17&16&14&13&11&10&9&8&7&6&5&4&2&1\\
			19&16&15&14&13&11&10&8&7&6&5&4&3&2&1\\
			10&7&6&5&4&2&1\\
			7&4&3&2&1\\
			2\\
			1}\end{array}
$$
$$\begin{array}{c} \textrm{The 3-partition constructed above, with cells from }\kappa \textrm{ highlighted.}\end{array}$$
\end{example}

\begin{remark}
In the proofs of Theorems \ref{top_and_bottom} and \ref{second_from_bottom} we will prove that a partition is an $\ell$-partition by giving its decomposition into $(\mu,r,\kappa)$.
\end{remark}

\subsection{Counting $\ell$-partitions}
We derive a closed formula for our generating function $B_{\ell}$ by using our $\ell$-partition decomposition described above. First we note that $$ \displaystyle x^{\ell-1} C_{\ell}(x) = \sum_{\mu \in \mathcal{C}_{\ell} \,:\, \mu_1-\mu_2 \,= \,\ell -1} x^{\mu_1}.$$ Therefore, $\displaystyle \sum_{\mu \in \mathcal{C}_{\ell} : \, \mu_1-\mu_2 \,\neq \,\ell-1} x^{\mu_1} = (1-x^{\ell-1}) C_{\ell}(x)$.  Hence the generating function for all cores $\mu$ with $\mu_1-\mu_2 \neq \ell-1$ is $\displaystyle \frac{1-x^{\ell-1}}{(1-x)^{\ell-1}}$.

We are now ready to describe the generating function for $\ell$-partitions with respect to the statistic of the first part.
Let $B_{\ell}(x) = \sum_{k=0}^{\infty} b_k^{\ell} x^k$ where $b_k^{\ell} = \# \{ \lambda \in \mathcal{L}_\ell: \lambda_1 = k \}$, i.e. $B_{\ell}(x) = \sum_{\lambda \in \mathcal{L}_{\ell}} x^{\lambda_1}$.

\begin{theorem} $$B_{\ell}(x) = \frac{1-x^{\ell-1}}{(1-x)^{\ell -1}(1-x^{\ell-1} - x^{\ell})}.$$
\end{theorem}
\begin{proof}
We will follow our construction of $\ell$-partitions from Section \ref{construct}. Note that if $\lambda \approx (\mu, r, \kappa)$, then the first part of $\lambda$ is ${\mu_1+\ell \kappa_1 + r (\ell-1)}$. 
Hence $\lambda$ contributes $x^{\mu_1+\ell \kappa_1 + r (\ell-1)}$ to $B_{\ell}$.

Fix a core $\mu$ with $\mu_1-\mu_2 \neq \ell-1$. Let $r$ and $\kappa_1$ be fixed non-negative integers. Let $\gamma_{r, \kappa_1}$ be the number of partitions with first part $\kappa_1$ and length less than or equal to $r+1$. $\gamma_{r, \kappa_1}$ counts the number of $\ell$-partitions with $r$ and $\kappa_1$ fixed that can be constructed from $\mu$. Note that $\gamma_{r, \kappa_1}$ is independent of what $\mu$ is. 

$\gamma_{r,\kappa_1}$ is the same as the number of partitions which fit inside a box of height $r$ and width $\kappa_1$. Hence $\gamma_{r, \kappa_1} = \binom{r+\kappa_1}{r}$. Fixing $\mu$ and $r$ as above, the generating function for the number of $\ell$-partitions with core $(\mu,r,\emptyset)$ with respect to the number of boxes added to the first row is $\displaystyle \sum_{\kappa_1=0}^{\infty} \gamma_{r,\kappa_1} x^{\kappa_1 \ell} = \frac{1}{(1-x^{\ell})^{r+1}}$.

Now for a fixed $\mu$ as above, the generating function for the number of $\ell$-partitions which can be constructed from $\mu$ with respect to the number of boxes added to the first row is $\displaystyle \sum_{r=0}^{\infty} x^{r (\ell-1)} \frac{1}{(1-x^{\ell})^{r+1}}$. Multiplying through by $1-x^{\ell-1} - x^{\ell}$, one can check that

 $$\sum_{r=0}^{\infty} x^{r (\ell-1)} \frac{1}{(1-x^{\ell})^{r+1}} = \frac{1}{1-x^{\ell-1} - x^{\ell}}.$$

 Therefore $B_{\ell}(x)$ is just the product of the two generating functions $(1-x^{\ell-1})C_{\ell}(x)$ and $\frac{1}{1-x^{\ell-1} - x^{\ell}}$. Hence $$B_{\ell}(x) = \frac{1-x^{\ell-1}}{(1-x)^{\ell -1}(1-x^{\ell-1} - x^{\ell})}.$$

\end{proof}

\begin{remark}
It would be desirable to obtain a formula for $\sum_{\lambda \in \mathcal{L}_\ell} x^{|\lambda|}$, but experimental evidence for $\ell = 2$ and $3$ showed this to be quite difficult.
\end{remark}

\subsection{Counting $\ell$-partitions of a fixed weight and fixed core} Independently of the authors, Cossey, Ondrus and Vinroot have a similar construction of partitions associated with irreducible representations. In \cite{COV}, they gave a construction analogous to our construction on $\ell$-partitions from Section \ref{construct} for the case of the symmetric group over a field of characteristic $p$. After reading their work and noticing the similarity to our own, we decided to include the following theorem, which is an analogue of their theorem for symmetric groups. The statement is a direct consequence of our construction, so no proof will be included.
\begin{theorem}
For a fixed core $\nu$ satisfying $\nu_i-\nu_{i+1} = \ell-1$ for $i = 1,2, \dots r$ and $\nu_{r+1}-\nu_{r+2} \neq \ell-1$, the number of $\ell$-partitions of a fixed weight $w$ is the number of  partitions of $w$ with at most $r+1$ parts. 
The generating function for the number of $\ell$-partitions of a fixed $\ell$-core $\nu$ with respect to the statistic of the weight of the partition is thus $\displaystyle \prod_{i=1}^{r+1} \frac{1}{1-x^i}$. Hence the generating function for all $(\ell,0)$-Carter partitions with fixed core $\nu$ with respect to the statistic of the size of the partition is 
$\displaystyle \sum_{\lambda\in \mathcal{L}_{\ell} \textrm{ with core } \nu} x^{|\lambda|} = x^{|\nu|} \prod_{i=1}^{r+1} \frac{1}{1-x^{\ell i}}$.

\end{theorem}
\begin{example}
Let $\ell = 3$ and let $\nu = (6,4,2,1,1) \approx ((2,1,1),2,\emptyset)$ be a 3-core. Then the number of 3-partitions of weight $5$ with core $\nu$ is exactly the number of partitions of $5$ into at most $3$ parts. There are 5 such partitions ($(5), (4,1), (3,2), (3,1,1), (2,2,1)$). Therefore, there are $5$ such $\ell$-partitions. They are: $$(21,4,2,1,1),  (18,7,2,1,1),  (15,10,2,1,1),  (15,7,5,1,1), (12,10,5,1,1).$$ 

$$\begin{array}{cc} \nu = & \tableau{\mbox{}&\mbox{}&\mbox{}&\mbox{}&\mbox{}&\mbox{}\\
\mbox{}&\mbox{}&\mbox{}&\mbox{}\\
\mbox{}&\mbox{}\\
\mbox{}\\
\mbox{}&
}
\end{array}$$

For $\nu$ above,  $r = 2$, so horizontal 3-rim hooks can be added to the first three rows.
\end{example}

\section{The crystal of the basic representation of $\widehat{\mathfrak{sl}_{\ell}}$}\label{crystal}

There is a crystal graph structure on the set of all $\ell$-regular partitions.
The crystal can be viewed as a $\Z/ \ell \Z$-colored directed graph whose nodes
are the $\ell$-regular partitions and whose outwardly oriented $i$-edges
indicate the addition of a particular box of residue $i$.
Representation-theoretically the nodes stand for irreducible representations
and the edges 
indicate a partial branching rule (the simple quotients of induction).

{\sloppy
Irreducible $H_n(q)$-modules  with interesting
representation-theoretic behavior often have nice combinatorial
characterizations in this crystal. 
An example is given by the irreducibles that
are projective $H_n(q)$-modules. 
These are precisely parameterized by the $\ell$-cores. 
They can be characterized crystal-theoretically as the nodes unique  with 
their given weight (part of the data that goes into the
definition of crystal, which is in this context the multiset of residues of a partition) or  as the extremal nodes as follows.
$\affS{\ell}$ acts on the nodes of the crystal (indeed on all partitions)
by ``reflecting $i$-strings" where an $i$-string is a 
maximal connected component 
of the subgraph consisting of just $i$-colored arrows.
The $\ell$-cores are the nodes in the $\affS{\ell}$-orbit of the highest weight node
(which is the unique node with no in-arrows), in this setting, the
empty partition. 

}
As we have seen in the previous sections, in some ways $\ell$-partitions
generalize $\ell$-cores.
It is then natural to expect that the combinatorial characterization of
$\ell$-partitions in the crystal is similar to that of $\ell$-cores.
And it is, but with a few crucial differences.
The $\ell$-partitions do not form an $\affS{\ell}$-orbit, nor
even a union of $\affS{\ell}$-orbits. 
 It is still true that if a node is an $\ell$-partition
then the extreme nodes in its $i$-string (for any $i \in \Z/\ell \Z$)
are also  $\ell$-partitions.
However, the $\ell$-partitions do not have to live just at 
the extremes. 
The condition can be relaxed, in some cases, to be
``second from the bottom" of an $i$-string, but
 nowhere else along the $i$-string, save the extreme ends. In particular the node second from the top of an $i$-string never corresponds to an $\ell$-partition except in the trivial cases that that node is coincidentally second from the bottom or at an extreme end.
This section gives a combinatorial proof of this fact, and 
Theorem  \ref{second_from_bottom} below characterizes 
when $\ell$-partitions are sub-extremal on an $i$-string.
At the moment we only have a partial representation-theoretic
explanation for the pattern.

\subsection{Description of crystal} We will assume some familiarity with the theory of crystals (see \cite{K}), and their relationship to the representation theory of the finite Hecke algebra (see \cite{G} or \cite{Kl}). We will look at the crystal $B(\Lambda_0)$ of the irreducible highest weight module $V(\Lambda_0)$ of the affine Lie algebra $\widehat{\mathfrak{sl}_{\ell}}$ (also called the basic representation of $\widehat{\mathfrak{sl}_{\ell}}$).
The set of nodes of $B(\Lambda_0)$ is denoted $B := \{ \lambda \in \mathcal{P} : \, \lambda \textrm{ is } \ell \textrm{-regular} \} $. We will describe the arrows of $B(\Lambda_0)$ below. This description is originally due to Misra and Miwa (see \cite{MM}).

We say the box $(a,b)$ of a partition has \textit{residue} $b-a \mod \ell$. A box $x$ in $\lambda$ is said to be a removable $i$-box if it has residue $i$ and after removing $x$ from $\lambda$ the remaining diagram is still a partition. A space $y$ not in $\lambda$ is an addable $i$-box if it has residue $i$ and adding $y$ to $\lambda$ yields a partition.

\begin{example}
Let $\lambda = (8,5,4,1)$ and $\ell = 3$. Then the residues are filled into the boxes of the corresponding Young diagram as follows:

$$\begin{array}{cc}
\lambda = & \tableau{0&1&2&0&1&2&0&1 \\
2&0&1&2&0 \\
1&2&0&1\\0}
\end{array}
$$
$\lambda$ has two removable 0-boxes (the boxes (2,5) and (4,1)), two removable 1-boxes (the boxes (1,8) and (3,4)), no removable 2-boxes, no addable 0-boxes, two addable 1-boxes (at (2,6) and (4,2)), and three addable 2-boxes (at (1,9), (3,5) and (5,1)).
\end{example}

For a fixed $i$, ($0 \leq i < \ell$), we place $-$ in each removable $i$-box and $+$ in each addable $i$-box. The $i$-signature of $\lambda$ is the word of $+$ and $-$'s in the diagram for $\lambda$, read from bottom left to top right. The reduced $i$-signature is the word obtained after repeatedly removing from the $i$-signature all pairs $- +$. The reduced $i$-signature is of the form $+ \dots +++--- \dots -$. The boxes corresponding to $-$'s in the reduced $i$-signature are called \textit{normal $i$-boxes}, and the positions corresponding to $+$'s are called \textit{conormal $i$-boxes}. $\varepsilon_i(\lambda)$ is defined to be the number of normal $i$-boxes of $\lambda$, and $\varphi_i(\lambda)$ is defined to be the number of conormal $i$-boxes. If there is at least one $-$ in the reduced $i$-signature, the box corresponding to the leftmost $-$ is called the \textit{good $i$-box} of $\lambda$. If there is at least one $+$ in the reduced $i$-signature, the position corresponding to the rightmost $+$ is called the \textit{cogood $i$-box}. All of these definitions can be found in Kleshchev's book \cite{Kl}.

\begin{example}
Let $\lambda = (8,5,4,1)$ and $\ell =3$ be as above. Fix $i=1$. The diagram for $\lambda$ with removable and addable 1-boxes marked looks like:
$$\tableau{ \mbox{} & \mbox{} & \mbox{}& \mbox{} & \mbox{}& \mbox{}& \mbox{}& -\\  \mbox{}& \mbox{}& \mbox{}&\mbox{} & \mbox{} \,\,\,\,\,\,\,\,\,\,\,\,\,\,\,\,\,\,\,\,\,+\\ \mbox{}& \mbox{}& \mbox{}& -\\ \mbox{}\,\,\,\,\,\,\,\,\,\,\,\,\,\,\,\,\,\,\,\,\,+ } $$

The 1-signature of $\lambda$ is $+-+-$, so the reduced 1-signature is $+ \,\,\,\,\,\,\,\,\,\,-$ and the diagram has a good 1-box in the first row, and a cogood 1-box in the fourth row. Here $\varepsilon_1(\lambda)=1$ and $\varphi_1(\lambda)=1$. 
\end{example}

We recall the action of the crystal operators on $B.$ The crystal operator $\widetilde{e}_{i}: B \xrightarrow{i} B \cup \{0\}$ assigns to a partition  $\lambda$ the partition $\widetilde{e}_{i}(\lambda) = \lambda \setminus x$, where $x$ is the good $i$-box of $\lambda$. If no such box exists, then $\widetilde{e}_{i}(\lambda)=0$. We remark that $\varepsilon_i(\lambda) = max\{k : \widetilde{e}_{i}^k \lambda \neq 0\}$.

Similarly, $\widetilde{f}_{i}: B \xrightarrow{i} B \cup \{0\}$ is the operator which assigns to a partition  $\lambda$ the partition $\widetilde{f}_{i}(\lambda) = \lambda \cup x$, where $x$ is the cogood $i$-box of $\lambda$. If no such box exists, then $\widetilde{f}_{i}(\lambda)=0$. We remark that $\varphi_i(\lambda) = max\{k : \widetilde{f}_{i}^k \lambda \neq 0\}$.

For $i$ in $\mathbb{Z}/ \ell \mathbb{Z}$, we write $\lambda \xrightarrow{i} \mu$ to stand for $\widetilde{f}_{i} \lambda = \mu$. We say that there is an $i$-arrow from $\lambda$ to $\mu$. Note that $\lambda \xrightarrow{i} \mu$ if and only if $\widetilde{e}_{i} \mu = \lambda$. A maximal chain of consecutive $i$-arrows is called an $i$-string. We note that the empty partition $\emptyset$ is the unique highest weight node of the crystal. For a picture of the first few levels of this crystal graph, see \cite{LLT} for the cases $\ell = 2$ and $3$.

\begin{example}
Continuing with the above example, $\widetilde{e}_{1} (8,5,4,1) = $ $(7,5,4,1)$ and $\widetilde{f}_{1}(8,5,4,1) = (8,5,4,2)$. Also, $\widetilde{e}_{1} ^2(8,5,4,1) = 0$ and $\widetilde{f}_{1}^2(8,5,4,1) = 0$. The sequence $(7,5,4,1) \xrightarrow{1} (8,5,4,1) \xrightarrow{1} (8,5,4,2)$ is a $1$-string of length 3.
\end{example}

\subsection{Crystal operators and $\ell$-partitions}
We first recall some well-known facts about the behavior of $\ell$-cores in this crystal graph $B(\Lambda_0)$. There is an action of the affine Weyl group $\widetilde{S_{\ell}}$ on the crystal such that the simple reflection $s_i$ reflects each $i$-string. In other words, $s_i$ sends a node $\lambda$ to 
$$\left\{ \begin{array}{lr}
\displaystyle
			  \widetilde{f}_{i}^{\varphi_i(\lambda) - \varepsilon_i(\lambda)} \lambda \;\;\;\;&  \varphi_i(\lambda) - \varepsilon_i(\lambda) >0\\
			  \widetilde{e}_{i}^{\varepsilon_i(\lambda) - \varphi_i(\lambda)} \lambda \;\;\;\;&  \varphi_i(\lambda) - \varepsilon_i(\lambda) <0\\
			  \lambda \;\;\;\; & \varphi_i(\lambda) - \varepsilon_i(\lambda) =0.
			\end{array} \right.$$ 
			
The set of $\ell$-cores is exactly the $\widetilde{S_{\ell}}$-orbit of $\emptyset$, the highest weight node. This implies the following Proposition.
\begin{proposition}\label{adding_to_core} If $\mu$ is an $\ell$-core and $\varepsilon_i(\mu) \neq 0$ then $\varphi_i(\mu) = 0$ and $\widetilde{e}_{i}^{\varepsilon_i(\mu)}\mu$ is again an $\ell$-core. Furthermore, $\widetilde{e}_{i}^{k}\mu$ is not an $\ell$-core for any $0<k<\varepsilon_i(\mu)$. Similarly, if $\varphi_i(\mu) \neq 0$ then $\varepsilon_i(\mu) = 0$ and $\widetilde{f}_{i}^{\varphi_i(\mu)}\mu$ is an $\ell$-core but $\widetilde{f}_i^{k} \mu$ is not for $0 < k < \varphi_i(\mu)$.
\end{proposition}

In this paper, given an $\ell$-partition $\lambda$, we will determine when $\widetilde{f}_{i}^{k}\lambda$ and $\widetilde{e}_i^k \lambda$ are also $\ell$-partitions.

The following remarks will help us in the proofs of the upcoming Theorems \ref{top_and_bottom}, \ref{other_cases} and \ref{second_from_bottom}. 
\begin{remark}\label{residue_remark} Suppose $\lambda$ is a partition. Consider its Young diagram. If any $\ell$-rim hook has an upper rightmost box of residue $i$, then the lower leftmost box has residue $i+1\mod \ell$. Conversely, a hook length $h_{(a,b)}^{\lambda}$ is divisible by $\ell$ if and only if there is an $i$ so that the rightmost box of row $a$ has residue $i$, and the lowest box of column $b$ has residue $i+1\mod \ell$.
\end{remark}
In Lemma \ref{adding_to_l_partition} we will generalize Proposition \ref{adding_to_core} to $\ell$-partitions.
\begin{proposition}\label{i_sig}
Let $\lambda$ be an $\ell$-core, and suppose $0\leq i <\ell$. Then the $i$-signature for $\lambda$ is the same as the reduced $i$-signature.
\end{proposition}

\begin{proof}
This follows from Remark \ref{residue_remark} above. 
\end{proof}
In particular, an $\ell$-core cannot have both a removable and an addable $i$-box.

\begin{lemma}\label{adding_to_l_partition} Let $\lambda$ be an $\ell$-partition, and suppose $0 \leq i < \ell$. Then the $i$-signature for $\lambda$ is the same as the reduced $i$-signature.
\end{lemma}

\begin{proof}
We need to show that there does not exist positions $(a,b)$ and $(c,d)$ such that $(a,b)$ is an addable $i$-box, $(c,d)$ is a removable $i$-box, and $c>a$. But if this were the case, then the hook length $h_{(a,d)}^\lambda$ would be divisible by $\ell$ (by Remark \ref{residue_remark}), but $\ell$ does not divide $h_{(c,d)}^\lambda =1$. Then $\lambda$ would violate $(\star)$, so it would not be an $\ell$-partition.
\end{proof}

\begin{remark}\label{adding_to_l_partition_remark}
As a consequence of Lemma \ref{adding_to_l_partition}, the action of the operators $\widetilde{e}_i$ and $\widetilde{f}_i$ is simplified in the case of $\ell$-partitions. For fixed $i$, applying successive $\widetilde{f}_{i}$'s to $\lambda$ corresponds to adding all addable boxes of residue $i$ from right to left (i.e. all addable $i$-boxes are conormal). Similarly, applying successive $\widetilde{e}_{i}$'s to $\lambda$ corresponds to removing all removable boxes of residue $i$ from left to right (i.e. all removable $i$-boxes are normal).
\end{remark}
In the following Theorems \ref{top_and_bottom}, \ref{other_cases} and \ref{second_from_bottom}, we implicitly use Remark \ref{residue_remark} to determine when a hook length is divisible by $\ell$, and  Remark \ref{adding_to_l_partition_remark} when applying $\widetilde{e}_{i}$ and $\widetilde{f}_{i}$ to $\lambda$. Unless it is unclear from the context, for the rest of the paper $\varphi = \varphi_i(\lambda)$ and $\varepsilon = \varepsilon_i(\lambda)$. 

\begin{remark} Suppose $\lambda \approx (\mu,r,\kappa)$. When viewing $\mu$ embedded in $\lambda$, we note that if a box $(a,b) \in \mu \subset \lambda$ has residue $i \mod \ell$ in $\lambda$, then it has residue $i-r \mod \ell$ in $\mu$. 
\end{remark}
Let $\lambda = (\lambda_1, \lambda_2, \dots )$ be a partition, and $r$ be any integer. We define $\overline{\lambda} = (\lambda_2, \lambda_3, \dots  )$, $\hat{\lambda} = (\lambda_1, \lambda_1, \lambda_2, \lambda_3, \dots )$ and $\lambda+1^r = (\lambda_1+1, \lambda_2+1, \dots, \lambda_r+1, \lambda_{r+1}, \dots )$, extending $\lambda$ by $r-len(\lambda)$ parts of size $0$ if $r > len(\lambda)$. We note that Lemma \ref{newcores} implies that $\overline{\lambda}$ is an $\ell$-core.

\subsection{$\ell$-partitions in the crystal $B(\Lambda_0)$}

\begin{theorem}\label{top_and_bottom}
Suppose that $\lambda$ is an $\ell$-partition and $0\leq i < \ell$. Then
\begin{enumerate}
\item\label{f} $\widetilde{f}_{i}^{\varphi} \lambda$ is an $\ell$-partition,
\item\label{e} $\widetilde{e}_{i}^{\varepsilon} \lambda$ is an $\ell$-partition.
\end{enumerate}
\end{theorem}

\begin{proof} We will prove only \eqref{f}, as \eqref{e} is similar. Recall all addable $i$-boxes of $\lambda$ are conormal by Lemma \ref{adding_to_l_partition}. The proof of \eqref{f} relies on the decomposition of the $\ell$-partition as in Section \ref{construct}.  Let $\lambda \approx (\mu, r, \kappa)$. We break the proof of \eqref{f} into three cases:
\begin{enumerate}
    \renewcommand{\labelenumi}{(\alph{enumi})}

\item If the first row of $\mu$ embedded in $\lambda$ does not have an addable $i$-box then we cannot add an $i$-box to the first $r+1$ rows of $\lambda$. Hence $\varphi = \varphi _{i-r}(\mu)$.  $\widetilde{f}_{{i-r}}^{\varphi} \mu$, 
is still a core by Proposition \ref{adding_to_core}. Hence we can exhibit the decomposition $\widetilde{f}_{i}^{\varphi} \lambda \approx ( \widetilde{f}_{{i-r}}^{\varphi} \mu, r, \kappa)$.

 \item If the first row of $\mu$ embedded in $\lambda$ does have an addable $i$-box and $\mu_1-\mu_2 < \ell-2$, then the first $r+1$ rows of $\lambda$ have addable $i$-boxes. Also some rows of $\mu$ will have addable $i$-boxes. $\widetilde{f}_i^{\varphi}$ adds an $i$-box to the first $r$ rows of $\lambda$, plus adds any addable $i$-boxes to the core $\mu$. Note that $\varphi_{i-r}(\mu) = \varphi -r$.  Since $\mu_1 - \mu_2 < \ell-2$, the first and second rows of $\widetilde{f}_{i-r}^{\varphi-r} \mu$ differ by at most $\ell-2$. Therefore $\widetilde{f}_{i}^{\varphi} \lambda \approx (\widetilde{f}_{{i-r}}^{\varphi-r} \mu, r, \kappa) $.

\item If the first row of $\mu$ embedded in $\lambda$ does have an addable $i$-box and $\mu_1-\mu_2 = \ell-2$, then $\widetilde{f}_i^{\varphi}$ will add the addable $i$-box in the $r+1^{st}$ row (i.e. the first row of $\mu$). Since the $(r+2)^{nd}$ row does not have an addable $i$-box, we know that the $(r+1)^{st}$ and $(r+2)^{nd}$ rows of $\widetilde{f}_i^{\varphi}(\lambda)$ differ by $\ell-1$. Therefore $\widetilde{f}_{i}^{\varphi} \lambda \approx (\overline{\widetilde{f}_{{i-r}}^{\varphi-r} \mu}, r+1, \kappa)$ is an $\ell$-partition.
 \end{enumerate}

\end{proof}

\begin{lemma}
Let $\lambda$ be an $\ell$-partition. Then $\lambda$ cannot have one normal box and two conormal boxes of the same residue.
\end{lemma}

\begin{proof}\label{lemmaforproofofothercases}
Label any two of the conormal boxes $n_1$ and $n_2$, with $n_1$ to the left of $n_2$. Pick any normal box and label it $n_3$.  By Lemma \ref{i_sig}, $n_3$ must lie to the right of $n_1$. Then the hook length in the column of $n_1$ and row of $n_3$ is a multiple of $\ell$, but the hook length in the column of $n_1$ and row of $n_2$ is not a multiple of $\ell$ by Remark \ref{residue_remark}.
\end{proof}

\begin{theorem}\label{other_cases}
Suppose that $\lambda$ is an $\ell$-partition. Then
\begin{enumerate}
\item\label{f_theorem} $\widetilde{f}_{i}^{k} \lambda$ is not an $\ell$-partition for $0<k < \varphi -1,$
\item\label{e_theorem} $\widetilde{e}_{i}^k \lambda$ is not an $\ell$-partition for $1<k< \varepsilon$.
\end{enumerate}
\end{theorem}
\begin{proof}

If $0< k < \varphi-1$ then there are at least two conormal $i$-boxes in $\widetilde{f}_{i}^{k} \lambda$ and at least one normal $i$-box.  By Lemma \ref{lemmaforproofofothercases}, $\widetilde{f}_{i}^{k} \lambda$ is not an $\ell$-partition. The proof of \eqref{e_theorem} is similar to that of \eqref{f_theorem}.
\end{proof}

The above theorems told us the position of an $\ell$-partition relative to the $i$-string which it sits on in the crystal $B(\Lambda_0)$. If an $\ell$-partition occurs on an $i$-string, then both ends of the $i$-string are also $\ell$-partitions. Furthermore, the only places $\ell$-partitions can occur are at the ends of $i$-strings or possibly one position before the final node. The next theorem describes when this latter case occurs.

\begin{theorem}\label{second_from_bottom}
Suppose that $\lambda \approx (\mu,r,\kappa)$ is an $\ell$-partition. Then

\begin{enumerate}
\item\label{first}
If $\varphi > 1$ then $\widetilde{f}_{i}^{\varphi-1} \lambda$ is an $\ell$-partition if and only if $$(\dagger) \,\,\,\,\,\,\,\,\,\,\,\kappa_{r+1} = 0, \textrm{ the first row of } \lambda \textrm{ has a conormal } i\textrm{-box, and } \varphi =r+1.$$

\item\label{second} If $\varepsilon > 1$ then $\widetilde{e}_{i} \lambda$ is an $\ell$-partition if and only if $$ (\ddagger) \textrm{ the first row of } \lambda \textrm{ has a conormal } \\(i+1) \textrm{-box and either }$$
 $$\varepsilon = r  \textrm{ and } \kappa_r = 0, \textrm{ or } \varepsilon = r+1 \textrm{ and } \kappa_{r+1} = 0 .$$
\end{enumerate}
\end{theorem}

\begin{proof}
We first prove \eqref{first} and then derive \eqref{second} from \eqref{first}. 

If $\lambda$ satisfies condition $(\dagger)$ then $\lambda$ differs from $\widetilde{f}_{i}^{\varphi-1} \lambda$ by one box in each of the first $r$ rows. Hence $(\widetilde{f}_{i}^{\varphi-1} \lambda)_r - (\widetilde{f}_{i}^{\varphi-1} \lambda)_{r+1}$ is a multiple of $\ell$, so that the first $r$ rows each have one more horizontal $\ell$-rim hook than they had in $\lambda$. After removing these horizontal $\ell$-rim hooks, we get the partition $(\widehat{\mu}, r-1, \emptyset)$. This decomposition is valid, as we will now show $\widehat{\mu}$ is an $\ell$-core. 
Since $\varphi_i(\mu)=1$, $\widetilde {f_i} \mu$ 
is also an $\ell$-core and so in particular $\ell \nmid
h_{(1,b)}^{\widetilde {f_i} \mu}$
for $1 \le b \le \mu_1+1$.
Note that $h_{(1,b)}^{\widehat  \mu} =
h_{(1,b)}^{\widetilde {f_i} \mu}=
h_{(1,b)}^{\mu} +1 $
for  $1 \le b \le \mu_1$, and for $a > 1$,
$h_{(a,b)}^{\widehat  \mu} =
h_{(a,b)}^{\mu}$, yielding $\ell \nmid h_{(a,b)}^{\widehat  \mu}$ for
all
boxes $(a,b) \in \widehat  \mu$.  By Remark \ref{divisibility},
$\widehat  \mu$ is an $\ell$-core. It is then easy to see that  $\widetilde{f}_{i}^{\varphi-1} \lambda \approx (\widehat{\mu}, r-1, \kappa+1^r)$, so therefore $\widetilde{f}_{i}^{\varphi-1} \lambda$ is an $\ell$-partition.

Conversely:
\begin{enumerate} 
    \renewcommand{\labelenumi}{(\alph{enumi})}
    
\item If the first part of $\lambda$ has a conormal $j$-box, with $j \neq i$, call this box $n_1$. If $j = i+1$ then the box $(r+1, \lambda_{r+1})$ has residue $i$. If an addable $i$-box exists, say at $(a,b)$, it must be below the first $r+1$ rows. But then the hook length $h_{(r+1,b)}^{\mu}$ is divisible by $\ell$. This implies that $\mu$ is not a core. So we assume $j \neq i+1$. Then $\widetilde{f}_{i}^{\varphi-1} \lambda$ has at least one normal $i$-box $n_2$ and exactly one conormal $i$-box $n_3$ with $n_3$ left of $n_2$ left of $n_1$. The hook length of the box in the column of $n_3$ and the row of $n_2$ is divisible by $\ell$, but the hook length of the box in the column of $n_3$ and the row of $n_1$ is not (by Remark \ref{residue_remark}). By Theorem \ref{maintheorem}, $\widetilde{f}_{i}^{\varphi-1} \lambda$ is not an $\ell$-partition. 

\item By (a), we can assume that the first row has a conormal $i$-box. If $\varphi \neq r+1$ then row $r+2$ of  $\widetilde{f}_{i}^{\varphi-1} \lambda$ will end in a $j$-box, for some $j \neq i$. Call this box $n_1$. Also let $n_2$ be any normal $i$-box in $\widetilde{f}_{i}^{\varphi-1} \lambda$ and $n_3$ be the unique conormal $i$-box. Then the box in the row of $n_1$ and column of $n_3$ has a hook length which is not divisible by $\ell$, but the box in the row of $n_2$ and column of $n_3$ has a hook length which is (by Remark \ref{residue_remark}). By Theorem \ref{maintheorem}, $\widetilde{f}_{i}^{\varphi-1} \lambda$ is not an $\ell$-partition.

\item Suppose $\kappa_{r+1} \neq 0$. By $(a)$ and $(b)$, we can assume that $\varphi = r+1$ and that the first row of $\lambda$ has a conormal $i$-box. Then the difference between $\lambda$ and $\widetilde{f}_{i}^{\varphi-1} \lambda = \widetilde{f}_i^r \lambda$ is an added box in each of the first $r$ rows. Remove $\kappa_r - \kappa_{r+1} +1$ horizontal $\ell$-rim hooks from row $r$ of $\widetilde{f}_i^{\varphi-1} \lambda$. Call the remaining partition $\nu$. Then $\nu_r = \nu_{r+1} = \mu_1 + \ell \kappa_{r+1}$. Hence a removable non-horizontal $\ell$-rim hook exists in $\nu$ taking the rightmost box from row $r$ with the rightmost $\ell-1$ boxes from row $r+1$. Thus $\widetilde{f}_{i}^{\varphi-1}(\lambda)$ is not an $\ell$-partition.
\end{enumerate}

To prove \eqref{second}, we note that by Theorem \eqref{other_cases} that if $\varphi \neq 0$ and $\varepsilon>1$ then $\widetilde{e}_{i} \lambda = \widetilde{e}_{i}^2 \widetilde{f}_{i} \lambda$ cannot be an $\ell$-partition. Hence we only consider $\lambda$ so that $\varphi_i(\lambda) = 0$. But then $\widetilde{f}_{i}^{\varphi_i(\widetilde{e}_{i}^{\varepsilon}(\lambda))-1} \widetilde{e}_{i}^{\varepsilon} \lambda = \widetilde{f}_{i}^{\varepsilon-1} \widetilde{e}_{i}^{\varepsilon} \lambda = \widetilde{e}_{i} \lambda$. From this observation, it is enough to show that $\lambda$ satisfies $(\ddagger)$ if and only if $\widetilde{e}_{i}^{\varepsilon} \lambda$ satisfies ($\dagger$). The proof of this follows a similar line as the above proofs, so it will be left to the reader.
\end{proof}

\begin{example} Fix $\ell = 3$. Let $\lambda = (9,4,2,1,1) \approx ((2,1,1), 2, (1))$.
$$
\begin{array}{cc} \lambda = &
\tableau{0&1&2&0&1&2&0&1&2 \\ 2&0&1&2\\1&2\\0\\2}
\end{array}$$ 
Here $\varphi_0 (\lambda) = 3$. $\widetilde{f}_{0} \lambda = (10,4,2,1,1)$ is not a $3$-partition, but $\widetilde{f}_{0}^{2} \lambda = (10,5,2,1,1) \approx ((2,2,1,1), 1, (2,1))$ and $\widetilde{f}_{0}^{3} \lambda = (10,5,3,1,1) \approx ((1,1), 3, (1))$ are $3$-partitions.
\end{example}

\section{A representation-theoretic proof of Theorem \ref{top_and_bottom}}\label{new_proof}
This proof relies heavily on the work of Grojnowski, Kleshchev et al. We recall some notation from \cite{G} but repeat very few definitions below. 

\subsection{Definitions and preliminaries}
In the category $Rep_n$ of finite-dimensional representations of the finite Hecke algebra $H_n(q)$, we define the Grothendieck group $K(Rep_n)$ to be the group generated by isomorphism classes of finite-dimensional representations, with relations $[\mathcal{M}_1] +[\mathcal{M}_3] = [\mathcal{M}_2]$ if there exists an exact sequence $0 \to \mathcal{M}_1 \to \mathcal{M}_2 \to \mathcal{M}_3 \to 0$. This is a finitely generated abelian group with generators corresponding to the irreducible representations of $H_n(q)$. The equivalence class corresponding to the module $\mathcal{M}$ is denoted $[\mathcal{M}]$.

Just as $S_n$ can be viewed as the subgroup of $S_{n+1}$ consisting of permutations which fix $n+1$, $H_n(q)$ can be viewed as a subalgebra of $H_{n+1}(q)$ (the generators $T_1, T_2, \dots ,T_{n-1}$ generate a subalgebra isomorphic to $H_n(q)$). Let $\mathcal{M}$ be a finite-dimensional representation of $H_{n+1}(q)$. Then it makes sense to view $\mathcal{M}$ as a representation of $H_n(q)$. This module is called the \textit{restriction of }$\mathcal{M}$\textit{ to } $H_n(q)$, and is denoted $Res^{H_{n+1}(q)}_{H_n(q)} \mathcal{M}$.
Similarly, we can define the induced representation of $\displaystyle \mathcal{M}$ by $ Ind^{H_{n+1}(q)}_{H_n(q)} \mathcal{M} = H_{n+1}(q) \otimes_{H_n(q)} \mathcal{M}$. Just as $S_b \subset S_a$, we can also consider $H_b(q) \subset H_a(q)$ and define corresponding restriction and induction functors.
To shorten notation, $Res^{H_{a}(q)}_{H_b(q)}$ will be written $Res_b^a$, and $Ind^{H_{a}(q)}_{H_b(q)}$ will be written as $Ind_b^{a}$.

If $\lambda$ and $\mu$ are partitions, it is said that $\mu$ covers $\lambda$, (written $\mu \succ \lambda$) if the Young diagram of $\lambda$ is contained in the Young diagram of $\mu$ and $|\mu| = |\lambda| +1$. 

The following proposition  is well known and can be found in \cite{M}. 
\begin{proposition}\label{branching_rule} Let $\lambda$ be a partition of $n$ and $S^{\lambda}$ be the Specht module corresponding to $\lambda$. Then $$\displaystyle [Ind^{n+1}_{n} S^{\lambda}] =  \sum_{\mu \succ \lambda} [S^{\mu}].$$
\end{proposition}
We consider functors $\widetilde{e}_{i}: Rep_n \to Rep_{n-1}$ and $\widetilde{f}_{i}: Rep_n \to Rep_{n+1}$ which commute with the crystal action on partitions in the following sense (see \cite{G} for definitions and details).
\begin{theorem}\label{tilda} Let $\lambda$ be an $\ell$-regular partition. Then:
\begin{enumerate}

\item $\widetilde{e}_{i} D^{\lambda} = D^{\widetilde{e}_{i} \lambda}$;
\item\label{f_tilda} $\widetilde{f}_{i} D^{\lambda} = D^{\widetilde{f}_{i} \lambda}$.
\end{enumerate}
\end{theorem}

We now consider the functors $f_i: Rep_n \to Rep_{n+1} $ and $e_i: Rep_n  \to Rep_{n-1} $ which refine induction and restriction (for a definition of these functors, especially in the more general setting of cyclotomic Hecke algebras, see \cite{G}). For a representation $\mathcal{M} \in Rep_n $ let $\varepsilon_i (\mathcal{M}) = max\{k : {e}_{i} ^k \mathcal{M} \neq 0\}$ and 
$\varphi_i (\mathcal{M}) = max\{k : {f}_{i}^k \mathcal{M} \neq 0\}$. Grojnowski concludes the following theorem.
\begin{theorem}\label{groj} Let $\mathcal{M}$ be a finite-dimensional representation of $H_n(q)$. Let $\varphi = \varphi_i (\mathcal{M})$ and $\varepsilon = \varepsilon_i (\mathcal{M})$.
\begin{enumerate}
\item\label{ind} $Ind_{n}^{n+1} \mathcal{M} =
 \bigoplus_i f_i  \mathcal{M}; \,\,\,\,\,\:\:\:\;\;\;\;\;\;  Res_{n}^{n+1} \mathcal{M} = 
 \bigoplus_i e_i \mathcal{M};$
\item\label{f^phi} $[f_i^{\varphi} \mathcal{M}] = \varphi ! [\widetilde{f}_i^{\varphi} \mathcal{M}]; \,\,\,\,\,\,\,\,\,\,\,\,\:\:\: \,\,\,\,\,\,\,\:\:\: \,\,\,\,\,\,\,\:\:\: [e_i^{\varepsilon} \mathcal{M}] = \varepsilon ! [\widetilde{e}_i^{\varepsilon} \mathcal{M}]. $ 
\end{enumerate}
\end{theorem}

For a module $D^{\mu}$ the \textit{central character} $\chi(D^{\mu})$ can be identified with the multiset of residues of the partition $\mu$. The following theorem allows us to define $\chi(S^{\mu})$ as well.
\begin{theorem}\label{central}
All composition factors of the Specht module $S^{\lambda}$ have the same central character.
\end{theorem}
\begin{theorem}\label{character}
$\chi(f_i(D^{\lambda})) = \chi(D^{\lambda}) \cup \{ i \}; \;\;\;\;\;\;\;\; \chi(e_i(D^{\lambda})) = \chi(D^{\lambda}) \setminus  \{ i \}$. 
\end{theorem}

We are now ready to present a representation-theoretic proof of Theorem \ref{top_and_bottom}, which states that if $\lambda$ is an $\ell$-partition lying anywhere on an $i$-string in the crystal $B(\Lambda_0)$, then the extreme ends of the $i$-string through $\lambda$ are also $\ell$-partitions.
\subsection{A representation-theoretic proof of Theorem \ref{top_and_bottom}}

\begin{proof}[Alternate Proof of Theorem \ref{top_and_bottom}] Suppose $\lambda$ is an $\ell$-partition and $| \lambda| = n$. Recall that by the result of James and Mathas \cite{JM} combined with Theorem \ref{maintheorem}, $S^{\lambda} = D^{\lambda}$ if and only if $\lambda$ is an $\ell$-partition. Let $F$ denote the number of addable $i$-boxes of $\lambda$ and let $\nu$ denote the partition corresponding to $\lambda$ plus all addable $i$-boxes.

First, we induce $S^{\lambda}$ from $H_n(q)$ to $H_{n+F}(q)$. Applying  Proposition \ref{branching_rule} $F$ times yields $$\displaystyle [Ind_n^{n+F} S^{\lambda}] = \sum_{\mu_{F} \succ \mu_{F-1} \succ \dots \succ \mu_1 \succ \lambda} [S^{\mu_{F}}].$$
Note $[S^{\nu}]$ occurs in this sum with coefficient $F !$ (add the $i$-boxes in any order), and everything else in this sum has different central character than $S^{\nu}$. Hence the direct summand which has the same central character as $S^{\nu}$ (i.e. the central character of $\lambda$ with $F$ more $i$'s) is $F! [S^{\nu}]$ in $K(Rep_{n+F})$. 

We next  
apply \eqref{ind} from Theorem \ref{groj} $F$ times to obtain 
$$[Ind_n^{n+F} D^{\lambda}] = \bigoplus_{i_1, \dots, i_{F}} [f_{i_1} \dots f_{i_{F}} {D^{\lambda}}].$$

\sloppy{The direct summand with central character $\chi(S^{\nu})$ is $[f_i^F D^{\lambda}]$ in $K(Rep_{n+F})$. Since $\lambda$ is an $\ell$-partition, $S^{\lambda} = D^{\lambda}$, so $Ind_n^{n+F} S^{\lambda} = Ind_n^{n+f} D^{\lambda}$ and we have shown that $F![S^{\nu}] = [f_i^F D^{\lambda}]$.

}

Since $S^{\lambda} = D^{\lambda}$ and $f_i^{F} D^{\lambda} \neq 0$, we know that $F \leq \varphi$. Similarly, since $Ind_n^{n+F+1} S^{\lambda}$ has no composition factors with central character $\chi(S^{\lambda}) \cup \{ \underbrace{i, i, \dots , i}_{F+1}\},$  we know that $F \geq \varphi$. Hence $F = \varphi$. 

By part \ref{f^phi} of Theorem \ref{groj},  $[(f_i)^{\varphi} D^{\lambda}] = \varphi ! [\widetilde{f}_i^{\varphi} D^{\lambda}]$. Then by Theorem \ref{tilda}, $[S^{\nu}] = [D^{\widetilde{f}_i^{\varphi}\lambda}]$. 

Since $F=\varphi$, $\nu = \widetilde{f}_i^{\varphi} \lambda$, so in particular $\nu$ is $\ell$-regular and $S^{\nu} = D^\nu$. 
Hence $\widetilde{f}_{i}^{\varphi} \lambda = \nu$ is an $\ell$-partition.

  The proof that $\widetilde{e}_{i}^{\varepsilon_i(\lambda)} \lambda$ is an $\ell$-partition follows similarly, with the roles of induction and restriction changed in Proposition \ref{branching_rule}, and the roles of $e_i$ and $f_i$ changed in Theorem \ref{groj}.

\end{proof}

We do not yet have representation-theoretic proofs of our other Theorems \ref{other_cases} and \ref{second_from_bottom}. We expect an analogue of Theorem \ref{other_cases}  to be true for the Hecke algebra over a field of arbitrary characteristic. In Theorem \ref{second_from_bottom} the conditions $(\dagger)$ and $(\ddagger)$ will change for different fields, so any representation-theoretic proof of this theorem should distinguish between these different cases.

\section{Related Literature}\label{conclusion}
We will end by mentioning some related work concerning $\ell$-partitions. Cossey, Ondrus and Vinroot (see \cite{COV}) have a construction for the case of the symmetric group in characteristic $p$ which is an analogue of our construction of $\ell$-partitions from Section \ref{construct}.
Fairly recent results of Fayers (\cite{F}) and Lyle (\cite{L}) give combinatorial conditions which characterize partitions $\lambda$ such that the corresponding Specht module $S^{\lambda}$ of $H_n(q)$ is irreducible when $q$ is a primitive $\ell^{th}$ root of unity, without the condition that $\lambda$ be $\ell$-regular and allowing the characteristic of the underlying field to be $p$. Such partitions are called $(\ell,p)$-JM partitions in \cite{F}, and can be viewed as $\ell$-singular analogues of $(\ell,0)$-Carter partitions. The PhD thesis of the first author presents results for $(\ell,0)$-JM partitions which are analogous to the theorems in this paper.

\bibliographystyle{amsalpha}

\end{document}